\documentclass[12pt,twoside]{amsart}

\usepackage{amsmath,amsthm,amscd,amssymb,mathrsfs,graphicx,amsfonts,mathrsfs}
\usepackage{amsfonts}
\usepackage{amssymb,enumerate}
\usepackage{amsthm}
\usepackage[all]{xy}
\usepackage{hyperref}
\allowdisplaybreaks[4] \footskip=12pt
\renewcommand{\uppercasenonmath}[1]{}

\numberwithin{equation}{section} \theoremstyle{plain}
\newtheorem*{thm*}{Main Theorem}
\newtheorem{thm}{Theorem}[section]
\newtheorem{cor}[thm]{Corollary}
\newtheorem*{cor*}{Corollary}
\newtheorem{lem}[thm]{Lemma}
\newtheorem*{lem*}{Lemma}

\newtheorem*{fact*}{Fact}

\newtheorem*{nota*}{Notation}
\newtheorem{prop}[thm]{Proposition}
\newtheorem*{prop*}{Proposition}
\newtheorem{rem}[thm]{Remark}
\newtheorem*{rem*}{Remark}

\newtheorem*{observation*}{Observation}

\newtheorem*{exa*}{Example}
\newtheorem{df}[thm]{Definition}
\newtheorem*{df*}{Definition}

\newtheorem*{con*}{Construction}

\newtheorem*{conj*}{Conjecture}

\renewcommand{\geq}{\geqslant}
\renewcommand{\leq}{\leqslant}

\begin{document}
\begin{center}
{\large  \bf Grade and Cohen-Macaulayness for DG-modules}

\vspace{0.5cm} Yuancheng Ning and Xiaoyan Yang\\
School of Science, Zhejiang University of Science and Technology, Hangzhou 310023\\
E-mail: yangxy@zust.edu.cn
\end{center}

\bigskip
\centerline { \bf  Abstract}
\leftskip10truemm \rightskip10truemm \noindent
We establish an inequality relating the
projective dimension of a DG-module in $\mathrm{D}^\mathrm{b}_\mathrm{f}(A)$ to its grade and introduce
the concept of perfect DG-modules as a natural generalization of perfect modules. It is proved that a DG-module $M$ over a local Cohen-Macaulay DG-ring with constant amplitude is Cohen-Macaulay if and only if $M$ is perfect and $\mathrm{amp}M \leq \mathrm{amp}\mathrm{R}\Gamma_{\bar{\mathfrak{m}}}(M)$. An affirmative answer is provided to Conjecture 2.11 of Yoshida [J. Pure Appl. Algebra 123 (1998) 313--326]. We also study the grade of DG-modules with finite injective dimension and examine the preservation of Cohen-Macaulayness under tensor products.
\leftskip10truemm \rightskip10truemm \noindent
\\[2mm]

\leftskip0truemm \rightskip0truemm
\section{\bf Introduction}\label{pre}

Let $R$ be a commutative noetherian ring and $M$ a finitely generated $R$-module. The \emph{grade}, $\mathrm{grade}_RM$, of $M$
was introduced by Rees as the least integer $t\geq0$ such that $\mathrm{Ext}^t_R(M,R)\neq0$. Rees \cite{R}
showed that $\mathrm{grade}_RM$ also coincides with the greatest length of an $R$-regular sequence contained in the annihilator $\mathrm{Ann}_R M$, which reveals a close connection with the classical notion of depth, $\mathrm{depth}_RM$.
The projective dimension, $\mathrm{projdim}_RM$, of $M$ is a fundamental invariant in homological algebra. As the groups $\mathrm{Ext}^t_R(M,R)$ can be computed from a projective resolution of $M$, one obviously has $\mathrm{grade}_RM\leq\mathrm{projdim}_RM$.

Research on grade and projective dimension frequently examines the structure of the set $\{i\geq0\hspace{0.02cm}|\hspace{0.02cm}\mathrm{Ext}^i_R(M,R)\neq0\}$. A central problem is to determine when this set consists of consecutive integers, specifically equal to the interval $[\mathrm{grade}_RM, \mathrm{projdim}_RM]$ (or more generally, equal to $[\mathrm{depth}R-\mathrm{depth}_RM, \text{some upper bound}]$)? For a commutative noetherian local ring $R$, the equality $$\mathrm{sup}\{i\geq0\hspace{0.02cm}|\hspace{0.02cm}\mathrm{Ext}^i_R(M,R)\neq0\}=\mathrm{depth}R-\mathrm{depth}_RM$$
has been established in many settings involving finite homological dimensions. This holds, for example, when
$M$ has finite projective dimension, or when
$R$ is Gorenstein. Relevant studies on this topic can be found in a series of works (e.g. \cite{AY,DAT,F,KY,Ya,Yo}).

A finitely generated $R$-module
$M$ is called \emph{perfect} if the equality $\mathrm{grade}_RM=\mathrm{projdim}_RM$ holds. Perfect modules play a important role in both homological algebra and algebraic geometry. A classical result states that over a local Cohen-Macaulay  ring, a finitely generated module is Cohen-Macaulay if and only if it is perfect. This reveals a profound equivalence between the homological property of perfectness and the algebra geometric property of being Cohen-Macaulay.

The concept of perfect modules has been successfully generalized. Foxby \cite{F} introduced the notion of G-perfect modules by replacing projective dimension with G-dimension. Jorge-P$\acute{\mathrm{e}}$rez, Martins and Mendoza-Rubio \cite{JMM} introduced the quasi-perfect modules by replacing projective dimension with quasi-projective dimension.  These allow for the development of a rich theory analogous to that of perfect modules over broader classes of rings (such as Gorenstein rings and Cohen-Macaulay rings), further extending the application of homological methods to the study of module and ring properties.

Yoshida \cite{Yo} investigated tensor products involving perfect modules and maximal Cohen-Macaulay modules. He showed that, over a local Cohen-Macaulay ring
$R$, such tensor products are always Cohen-Macaulay. More generally, the same conclusion holds provided the Codimension Conjecture is true, that is, for a local noetherian ring
$R$ and a finitely generated
$R$-module $M$
 with finite projective dimension, the equality $\mathrm{dim}_RM+\mathrm{grade}_RM=\mathrm{dim}R$ holds. Yoshida subsequently proposed the following conjecture.

\vspace{1.5mm} \noindent{\bf Conjecture.}\label{lem:0.3} Let $R$ be a local noetherian ring  and $M$ a perfect $R$-module of positive dimension and $N$ a finitely generated $R$-module with $\mathrm{dim}_RN=\mathrm{dim}R$. If $M\otimes_RN$ is Cohen-Macaulay, then $N$ is a maximal Cohen-Macaulay $R$-module.
\vspace{1.5mm}

DG-rings (short
for differential graded rings) allow
us to use techniques of  homological algebra of ordinary rings in a much wider context.
The theory of local Cohen-Macaulay DG-rings offers a powerful generalization of classical homological and commutative algebra. It allows the study of fundamental properties, such as depth, regular sequences and singularities of rings in a more general setting. The systematic study of local Cohen-Macaulay DG-rings and modules was initiated by Shaul in \cite{Shaul,sh20}.  He demonstrated that much of the classical theory of Cohen-Macaulay rings and modules can be generalized to the DG framework,  also provided many natural examples of local Cohen-Macaulay DG-rings,  such as sequence-regular DG-rings and local Gorenstein DG-rings. In subsequent work \cite{sh21}, Shaul further developed some results on Cohen-Macaulay DG-rings, enriching the theory of these structures within homological algebra and algebraic geometry.

This paper focuses on the grade and the Cohen-Macaulayness for DG-modules, is organized as follows. In Section 2, we provide definitions and some results that are considered in this paper. In Section 3, we establish an inequality between the grade and the projective dimension of DG-modules and introduce the concept of perfect DG-modules. Within this framework, several classical results on perfect modules are extended to the setting of perfect DG-modules. We prove that for a DG-module $M$ over a local Cohen-Macaulay DG-ring with constant amplitude, it is Cohen-Macaulay if and only if it is perfect and satisfies $\mathrm{amp}M\leq\mathrm{amp}\mathrm{R}\Gamma_{\bar{\mathfrak{m}}}(M)$ (see Theorem \ref{lem2.5}). Furthermore, we give an affirmative answer to the above conjecture (see Theorem \ref{lem3.4}). As applications, we present several Cohen-Macaulay criterias for DG-rings and DG-modules, and provide a partial answer to the Codimension Conjecture of Auslander. In Section 4, we study the grade of DG-modules of finite injective dimension and obtain the formula $\mathrm{lc.dim}_AM=\mathrm{depth}A-\mathrm{grade}_AM$ for any
 $M\in\mathrm{D}^\mathrm{b}_\mathrm{f}(A)$ with finite injective dimension (see Theorem \ref{lem2.6}). Moreover, by utilizing the Cohen-Macaulayness of the Koszul DG-module $M/\hspace{-0.15cm}/\bar{\emph{\textbf{x}}}$  associated with a finite sequence $\bar{\emph{\textbf{x}}}=\bar{x}_1,\cdots,\bar{x}_n$ in $A$, we examine the preservation of Cohen-Macaulayness under tensor products in the DG context (see Theorem \ref{lem2.3} and Corollary \ref{lem1.6}).

\bigskip
\section{\bf Preliminaries}
\vspace{2mm}
This section collects some preliminaries that will be used throughout this paper. For terminology we shall follow \cite{Y} and \cite{Shaul}.

\vspace{1.5mm}
{\bf DG-modules.} An \emph{associative DG-ring} $A$ is a $\mathbb{Z}$-graded ring
$A=\bigoplus_{i\in\mathbb{Z}}A^i$
equipped with a $\mathbb{Z}$-linear map $d:A\rightarrow A$ of degree +1 such that $d\circ d=0$, and such that
\begin{equation}\label{eq:2.1}d(ab)=d(a)b+(-1)^{i}ad(b)
\end{equation}
for $a\in A^i$ and $b\in A^j$.
 A DG-ring $A$ is
called \emph{commutative} if $b\cdot a=(-1)^{i\cdot j}ab$ for $a\in A^i$ and $b\in A^j$, and $a^2=0$
if $i$ is odd.
A DG-ring $A$ is called \emph{non-positive} if $A^i=0$ for all $i>0$.
 A non-positive DG-ring $A$ is called \emph{noetherian} if the
ring $\mathrm{H}^0(A)$ is noetherian and the $\mathrm{H}^0(A)$-module $\mathrm{H}^i(A)$ is
finitely generated for all $i<0$.
If $A$ is a noetherian DG-ring and $(\mathrm{H}^0(A),\bar{\mathfrak{m}},\bar{\kappa})$ is a local ring, then we say that $(A,\bar{\mathfrak{m}},\bar{\kappa})$ is \emph{local noetherian}.
\textbf{Unless stated to the contrary we assume throughout this paper that $A$ is a commutative and noetherian DG-ring with $\mathrm{amp}A<\infty$.}

A DG-module $M$ is a graded $A$-module together with a differential
of degree +1 satisfying the Leibniz rule similar to (2.1).
 For a DG-module $M$, set $\mathrm{inf}M:=\mathrm{inf}\{n\in\mathbb{Z}\hspace{0.03cm}|\hspace{0.03cm}\mathrm{H}^n(M)\neq0\}$,
$\mathrm{sup}M:=\mathrm{sup}\{n\in\mathbb{Z}\hspace{0.03cm}|\hspace{0.03cm}\mathrm{H}^n(M)\neq0\}$
and $\mathrm{amp}M:=\mathrm{sup}M-\mathrm{inf}M$. The $i$th shift of $M$ is
denoted  by $M[i]$ for $i\in\mathbb{Z}$.
The derived category of DG-modules over $A$ is denoted by $\mathrm{D}(A)$, and write $-\otimes_A^\mathrm{L}-$
 for its monoidal product, $\mathrm{RHom}_A(-,-)$ for the internal hom. The full subcategory which consists of DG-modules
whose cohomology is bounded above (resp. bounded below, bounded) is denoted by $\mathrm{D}^-(A)$ (resp.
$\mathrm{D}^+(A)$, $\mathrm{D}^\mathrm{b}(A)$), denote by $\mathrm{D}_\mathrm{f}(A)$ the full subcategory of $\mathrm{D}(A)$
consisting of DG-modules with finite cohomology, and set
$\mathrm{D}^-_\mathrm{f}(A)=\mathrm{D}^-(A)\cap \mathrm{D}_\mathrm{f}(A)$, $\mathrm{D}^+_\mathrm{f}(A)=\mathrm{D}^+(A)\cap \mathrm{D}_\mathrm{f}(A)$, $\mathrm{D}^\mathrm{b}_\mathrm{f}(A)=\mathrm{D}^\mathrm{b}(A)\cap \mathrm{D}_\mathrm{f}(A)$.

We recall the definition of the projective, injective and flat dimensions of $M\in \mathrm{D}(A)$ introduced by Bird, Shaul, Sridhar and Williamson in \cite{BSSW}.

The \emph{projective dimension} of $M$ is defined by
\begin{center}$\mathrm{projdim}_AM=\mathrm{inf}\{n\in\mathbb{Z}\hspace{0.03cm}|\hspace{0.03cm}\mathrm{H}^i(\mathrm{RHom}_A(M,Y))=0\ \textrm{for\ any}\ Y\in\mathrm{D}^\mathrm{b}(A)\ \textrm{and}\ i>n+\mathrm{sup}Y\}$.\end{center}

The \emph{injective dimension} of $M$ is defined by
\begin{center}$\mathrm{injdim}_AM=\mathrm{inf}\{n\in\mathbb{Z}\hspace{0.03cm}|\hspace{0.03cm}\mathrm{H}^i(\mathrm{RHom}_A(Y,M))=0\ \textrm{for\ any}\ Y\in\mathrm{D}^\mathrm{b}(A)\ \textrm{and}\ i>n-\mathrm{inf}Y\}$.\end{center}

The \emph{flat dimension} of $M$ is defined by
$$\mathrm{flatdim}_AM=\mathrm{inf}\{n\in\mathbb{Z}\hspace{0.03cm}|\hspace{0.03cm}\mathrm{H}^{-i}(Y\otimes_A^\mathrm{L}M)=0\ \textrm{for\ any}\ Y\in\mathrm{D}^\mathrm{b}(A)\ \textrm{and}\ i>n-\mathrm{inf}Y\}.$$

{\bf Local Cohen-Macaulay DG-modules.} Let $\bar{\mathfrak{a}}\subseteq\mathrm{H}^0(A)$ be an ideal.
An $\mathrm{H}^0(A)$-module $\bar{M}$ is called \emph{$\bar{\mathfrak{a}}$-torsion} if for any $\bar{x}\in\bar{M}$ there is $n\in\mathbb{N}$ such that $\bar{\mathfrak{a}}^n\bar{x}=0$.
The category $\mathrm{D}_{\bar{\mathfrak{a}}\textrm{-tor}}(A)$ consisting of DG-modules $M$ so that
each $\mathrm{H}^0(A)$-module $\mathrm{H}^n(M)$ is $\bar{\mathfrak{a}}$-torsion. Following \cite{s19}, the inclusion functor
$\mathrm{D}_{\bar{\mathfrak{a}}\textrm{-tor}}(A)\hookrightarrow\mathrm{D}(A)$
has a right adjoint,
and composing this right adjoint with the inclusion, one has a triangulated
functor
$\mathrm{R}\Gamma_{\bar{\mathfrak{a}}}:\mathrm{D}(A)\rightarrow\mathrm{D}(A)$,
which is called the \emph{derived $\bar{\mathfrak{a}}$-torsion functor}.  The derived $\bar{\mathfrak{a}}$-adic completion $\mathrm{RHom}_A(\mathrm{R}\Gamma_{\bar{\mathfrak{a}}}(A),A)$, denoted by $\mathrm{L}\Lambda(A,\bar{\mathfrak{a}})$, is a commutative DG-ring.

Let $(A,\bar{\mathfrak{m}},\bar{\kappa})$ be a local DG-ring. Shaul \cite{Shaul} extended the definitions of Krull dimensions and depth of complexes to DG-setting. The \emph{local cohomology Krull dimension}
of a DG-module $M\in\mathrm{D}^-(A)$ is the number
\begin{center}$\mathrm{lc.dim}_AM:=
\mathrm{sup}\{\mathrm{dim}_{\mathrm{H}^0(A)}\mathrm{H}^\ell(M)+\ell\hspace{0.03cm}|\hspace{0.03cm}\ell\in\mathbb{Z}\}$,\end{center}where $\mathrm{dim}_{\mathrm{H}^0(A)}\mathrm{H}^\ell(M)$ is the usual Krull dimension of the $\mathrm{H}^0(A)$-module $\mathrm{H}^\ell(M)$.
 The \emph{depth} of a DG-module $N\in\mathrm{D}^+(A)$ is the number
\begin{center}$\mathrm{depth}_AN:=\mathrm{inf}\mathrm{RHom}_A(\bar{\kappa},N)$.\end{center}
Following \cite{Shaul},  a DG-module $M$ in $\mathrm{D}^{\mathrm{b}}_{\mathrm{f}}(A)$ is called
\emph{local Cohen-Macaulay} if
$$\mathrm{amp}M = \mathrm{amp}A = \mathrm{amp}\mathrm{R}\Gamma_{\bar{\mathfrak{m}}}(M)=\mathrm{lc.dim}_AM-\mathrm{depth}_AM,$$where the last equality is by \cite[Theorem 2.15]{Shaul} and \cite[Proposition 3.3]{Shaul}. A local Cohen-Macaulay DG-module $M$ is called \emph{maximal local Cohen-Macaulay} if
$\mathrm{lc.dim}_AM= \sup M +\mathrm{lc.dim}A$. A local DG-ring $A$ is called \emph{local Cohen-Macaulay} if $\mathrm{amp}A = \mathrm{amp}\mathrm{R}\Gamma_{\bar{\mathfrak{m}}}(A)$. A local DG-ring $A$ is called \emph{Gorenstein} if $\mathrm{injdim}_AA<\infty$.

\vspace{1.5mm}
{\bf Localization.} Let $\pi_A:A\rightarrow\mathrm{H}^0(A)$ be the canonical surjection and $\pi^0_A: A^0\rightarrow\mathrm{H}^0(A)$ be its degree 0 component. Given a prime ideal
$\bar{\mathfrak{p}}\in \mathrm{Spec}\mathrm{H}^0(A)$, let $\mathfrak{p}=(\pi^0_A)^{-1}(\bar{\mathfrak{p}})\in\mathrm{Spec}A^0$, and define $A_{\bar{\mathfrak{p}}}:=A\otimes_{A^0}A^0_\mathfrak{p}$. More generally, given $X\in\mathrm{D}(A)$, we define
\begin{center}$M_{\bar{\mathfrak{p}}}:=M\otimes_{A}A_{\bar{\mathfrak{p}}}=M\otimes_{A^0}A^0_\mathfrak{p}\in\mathrm{D}(A_{\bar{\mathfrak{p}}})$.\end{center} A prime ideal $\bar{\mathfrak{p}}\in\mathrm{Spec}\mathrm{H}^0(A)$ is called an \emph{associated
prime} of $M$ if $\mathrm{depth}_{A_{\bar{\mathfrak{p}}}}M_{\bar{\mathfrak{p}}}=\mathrm{inf}M_{\bar{\mathfrak{p}}}$. The set of associated primes of $M$ is denoted by $\mathrm{Ass}_AM$. The \emph{support} of $M$ is the set
$\mathrm{Supp}_AM:=\{\bar{\mathfrak{p}}\in \mathrm{Spec}\mathrm{H}^0(A)\hspace{0.03cm}|\hspace{0.03cm}M_{\bar{\mathfrak{p}}}\not\simeq0\}$.
Following \cite{sh20,YL}, we say that $M\in \mathrm{D}^{\mathrm{b}}_{\mathrm{f}}(A)$ has \emph{constant amplitude} if $\mathrm{amp}M_{\bar{\mathfrak{p}}}=\mathrm{amp}M$ for all $\bar{\mathfrak{p}}\in\mathrm{Supp}_AM$.

\vspace{1.5mm}
{\bf Sequence-regular DG-rings.} Let $(A,\bar{\mathfrak{m}},\bar{\kappa})$ be a local DG-ring and $M\in\mathrm{D}^+(A)$. An element $\bar{x}\in\bar{\mathfrak{m}}$ is called \emph{$M$-regular} if it is
$\mathrm{H}^{\mathrm{inf}M}(M)$-regular, i.e. the multiplication map
\begin{center}$\bar{x}:\mathrm{H}^{\mathrm{inf}M}(M)\rightarrow\mathrm{H}^{\mathrm{inf}M}(M)$\end{center}is injective. Inductively, a sequence $\bar{x}_1,\cdots,\bar{x}_n\in\bar{\mathfrak{m}}$ is \emph{$M$-regular} if
$\bar{x}_1$ is $M$-regular and the sequence $\bar{x}_2,\cdots,\bar{x}_n$ is $M/\hspace{-0.15cm}/\bar{x}_1M$-regular. Here, $M/\hspace{-0.15cm}/\bar{x}_1M$ is the
cone of the map $\bar{x}_1:M\rightarrow M$ in $\mathrm{D}(A)$.
For $M\in\mathrm{D}^\mathrm{b}_\mathrm{f}(A)$, $\mathrm{sup}M=\mathrm{sup}M/\hspace{-0.15cm}/\bar{x}M$ by Nakayama's lemma. By \cite[Lemma 2.13]{Mi19}, $\bar{x}$ is $M$-regular if and only if $\mathrm{amp}M=\mathrm{amp}M/\hspace{-0.15cm}/\bar{x}M$.
Following \cite{sh21}, a local DG-ring $(A,\bar{\mathfrak{m}},\bar{\kappa})$ is called  \emph{sequence-regular} if $\bar{\mathfrak{m}}$ is generated by an
$A$-regular sequence.

\vspace{1.5mm}The following lemma is frequently used through the article.

\begin{lem}\label{lem0.15} {\rm (\cite{Mi19,Y25,ya20}).} $(1)$ Let $M\in\mathrm{D}^+(A),N\in\mathrm{D}^\mathrm{b}_\mathrm{f}(A)$ with $\mathrm{projdim}_AN<\infty$, then
\begin{center}$\mathrm{depth}_A(M\otimes^\mathrm{L}_AN)=\mathrm{depth}_AM+\mathrm{depth}_AN-\mathrm{depth}A$.\end{center}

$(2)$ Let $M,N\in\mathrm{D}^{-}_{\mathrm{f}}(A)$ with $\mathrm{projdim}_AM<\infty$, one has an inequality
 \begin{center}$\mathrm{lc.dim}_AN\leq\mathrm{lc.dim}_A(M\otimes^\mathrm{L}_AN)+\mathrm{projdim}_AM$.\end{center}

$(3)$ Let $M\in\mathrm{D}^-_\mathrm{f}(A)$ with $\mathrm{projdim}_AM<\infty$,
then $\mathrm{projdim}_AM=\mathrm{depth}A-\mathrm{depth}_AM$.

$(4)$ Let $N\in\mathrm{D}^+_\mathrm{f}(A)$ with $\mathrm{injdim}_AN<\infty$, then
$\mathrm{injdim}_AN-\mathrm{sup}N=\mathrm{depth}A$.
\end{lem}

\bigskip
\section{\bf Perfect and local Cohen-Macaulay DG-modules}
In this section, we introduce the notions of grade and perfection for DG-modules and establish their DG analogues of several classical results concerning perfect modules.
Furthermore, we study the Cohen-Macaulay property of tensor products of DG-modules and give an affirmative answer to the conjecture in the introduction.

\begin{df}\label{lem:2.1} {\rm (1) The \emph{grade} of $M\in \mathrm{D}^{\mathrm{b}}_{\mathrm{f}}(A)$,
$\mathrm{grade}_AM$, is defined as the formula $$\mathrm{grade}_AM:=\mathrm{inf}\mathrm{RHom}_A(M,A).$$

(2)   A DG-module $M\in\mathrm{D}^{\mathrm{b}}_{\mathrm{f}}(A)$ is called \emph{perfect} if $\mathrm{grade}_AM=\mathrm{projdim}_AM+\mathrm{inf}A$.}
\end{df}
By the proof of \cite[Lemma 2.3(1)]{Y25}, one has an equality
\begin{align}
\mathrm{grade}_AM=\mathrm{inf}\{\mathrm{depth}A_{\bar{\mathfrak{p}}}
-\mathrm{sup}M_{\bar{\mathfrak{p}}}\hspace{0.02cm}|\hspace{0.02cm}\bar{\mathfrak{p}}\in\mathrm{Supp}_AM\}.
\label{exact03}
\tag{$\ast$}\end{align}
The next lemma establish some inequalities concerning the grade of DG-modules.

\begin{lem}\label{lem2.4} Let $A$ be a local DG-ring and $M\in\mathrm{D}^{\mathrm{b}}_{\mathrm{f}}(A)$. Then

$(1)$ $\mathrm{grade}_AM\geq\mathrm{depth}A-\mathrm{lc.dim}_AM$.

$(2)$ If $A$ has constant amplitude, then $\mathrm{grade}_AM\leq\mathrm{lc.dim}A-\mathrm{lc.dim}_AM+\mathrm{inf}A$. In additional, if
$A$ is local Cohen-Macaulay, then $\mathrm{grade}_AM=\mathrm{depth}A-\mathrm{lc.dim}_AM$.

$(3)$ $\mathrm{grade}_AM\leq\mathrm{projdim}_AM$. In additional, if $A$ has constant amplitude, then $\mathrm{grade}_AM\leq\mathrm{projdim}_AM+\mathrm{inf}A$.
\end{lem}
\begin{proof} (1) By the equality $(\ast)$, choose $\bar{\mathfrak{p}}\in\mathrm{Spec}\mathrm{H}^0(A)$ such that $\mathrm{grade}_AM=\mathrm{depth}A_{\bar{\mathfrak{p}}}
-\mathrm{sup}M_{\bar{\mathfrak{p}}}$. By  \cite[Lemma 2.26]{Mi19}, $\mathrm{depth}A\leq\mathrm{depth}A_{\bar{\mathfrak{p}}}+\mathrm{dim}\mathrm{H}^0(A)/\bar{\mathfrak{p}}$. By the equality (2.14) in \cite[Proposition 2.13]{Shaul}, $\mathrm{lc.dim}_AM\geq\mathrm{dim}\mathrm{H}^0(A)/\bar{\mathfrak{p}}+
\mathrm{sup}M_{\bar{\mathfrak{p}}}$. Thus $\mathrm{depth}A-\mathrm{lc.dim}_AM\leq\mathrm{depth}A_{\bar{\mathfrak{p}}}+\mathrm{dim}\mathrm{H}^0(A)/\bar{\mathfrak{p}}-(\mathrm{dim}\mathrm{H}^0(A)/\bar{\mathfrak{p}}+
\mathrm{sup}M_{\bar{\mathfrak{p}}})=\mathrm{grade}_AM$.

(2) By \cite[Equality (2.14)]{Shaul}, choose $\bar{\mathfrak{q}}\in\mathrm{Spec}\mathrm{H}^0(A)$ such that $\mathrm{lc.dim}_AM=\mathrm{dim}\mathrm{H}^0(A)/\bar{\mathfrak{q}}+
\mathrm{sup}M_{\bar{\mathfrak{q}}}$. One has the following inequalities: \begin{center}$\begin{aligned}\mathrm{grade}_AM
&\leq\mathrm{depth}A_{\bar{\mathfrak{q}}}-\mathrm{sup}M_{\bar{\mathfrak{q}}}\\
&\leq\mathrm{lc.dim}A_{\bar{\mathfrak{q}}}+\mathrm{dim}\mathrm{H}^0(A)/\bar{\mathfrak{q}}
+\mathrm{inf}A_{\bar{\mathfrak{q}}}-\mathrm{dim}\mathrm{H}^0(A)/\bar{\mathfrak{q}}-\mathrm{sup}M_{\bar{\mathfrak{q}}}\\
&\leq\mathrm{lc.dim}A+\mathrm{inf}A-\mathrm{lc.dim}_AM,\end{aligned}$\end{center}where the first one is by the equality $(\ast)$, the second one is by \cite[Corollary 5.5]{Shaul} and the third one is by \cite[Lemma 5.10]{Shaul} and $\mathrm{inf}A_{\bar{\mathfrak{q}}}=\mathrm{inf}A$.

(3) By \cite[Theorem 3.4]{ya20}, one has $\mathrm{grade}_AM\leq\mathrm{projdim}_AM$. Assume that $A$ has constant amplitude. We may assume that $\mathrm{projdim}_AM<\infty$, it follows by \cite[Corollary 3.3]{Y25} and Lemma \ref{lem0.15}(3) that $\mathrm{lc.dim}A-\mathrm{lc.dim}_AM\leq\mathrm{depth}A-\mathrm{depth}_AM=\mathrm{projdim}_AM$, the desired inequality follows by (2).
\end{proof}

 \begin{rem}\label{lem:2.2}{\rm (1) If $A$ is an ordinary ring and $M$ is a finitely generated $A$-module, then this definition coincides with the usual definition of grade of $M$ defined by Rees \cite{R}.

 (2) If $M$ is perfect, then $\mathrm{ampRHom}_A(M,A)=\mathrm{projdim}_AM-\mathrm{grade}_AM=\mathrm{amp}A$ by \cite[Proposition 4.4(3)]{ya20}.

 (3) If $A$ is local Gorenstein and $M\in\mathrm{D}^{\mathrm{b}}_{\mathrm{f}}(A)$ with $\mathrm{amp}M=\mathrm{amp}A$, then $M$ is perfect if and only if $M$ is local Cohen-Macaulay by (2) and \cite[Proposition 6.2]{Shaul}.

 (4) If $A$ has constant amplitude and $M$ is local Cohen-Macaulay with $\mathrm{projdim}_AM<\infty$, then $\mathrm{projdim}_AM+\mathrm{inf}A=\mathrm{depth}A-\mathrm{depth}_AM+\mathrm{inf}A=\mathrm{depth}A-\mathrm{lc.dim}_AM\leq\mathrm{grade}_AM\leq\mathrm{projdim}_AM+\mathrm{inf}A$ by Lemma \ref{lem2.4}, it implies that $M$ is perfect.}
\end{rem}

Following \cite{FIJ}, a \emph{dualizing DG-module} $R$ over a DG-ring $A$ is a DG-module $R\in\mathrm{D}^{+}_{\mathrm{f}}(A)$ such that $\mathrm{injdim}_AR<\infty$ and the natural map $A\rightarrow\mathrm{RHom}_A(R,R)$ is an isomorphism in $\mathrm{D}(A)$. We say that $R$ is \emph{normalized} if $\mathrm{inf}R=-\mathrm{lc.dim}A$. The
following result provides a Cohen-Macaulay criterion.

\begin{thm}\label{lem2.5} Let $A$ be a local Cohen-Macaulay DG-ring with constant amplitude and $M\in\mathrm{D}^{\mathrm{b}}_{\mathrm{f}}(A)$ with $\mathrm{projdim}_AM<\infty$. The following are equivalent:

$(1)$ $M$ is local Cohen-Macaulay;

 $(2)$ $M$ is perfect and $\mathrm{amp}M\leq\mathrm{amp}\mathrm{R}\Gamma_{\bar{\mathfrak{m}}}(M)$.\\
 In additional, if $A$ has a dualizing DG-module $R$ then the above conditions are equivalent to

 $(3)$ $M$ is perfect and $\mathrm{RHom}_A(M,A)$ is local Cohen-Macaulay.
\end{thm}
\begin{proof} (1) $\Rightarrow$ (2) This follows by  Remark \ref{lem:2.2}(4).

(2) $\Rightarrow$ (1) By Lemma \ref{lem2.4}(2), $\mathrm{grade}_AM=\mathrm{depth}A-\mathrm{lc.dim}_AM$. By Lemma \ref{lem0.15}(3), $\mathrm{projdim}_AM=\mathrm{depth}A-\mathrm{depth}_AM$, it follows by \cite[Proposition 4.1]{Y25} that
$-\mathrm{inf}A\leq\mathrm{amp}M\leq\mathrm{amp}\mathrm{R}\Gamma_{\bar{\mathfrak{m}}}(M)=\mathrm{projdim}_AM-\mathrm{grade}_AM=-\mathrm{inf}A$. Thus $M$ is local Cohen-Macaulay.

We may assume that $R$ is normalized. As $A$ has constant amplitude, so is  $R$ by \cite[Theorem 4.1(2) and Proposition 4.4]{Shaul}, it follows by \cite[Page138]{B} that \begin{center}$\mathrm{Ass}_{\mathrm{H}^0(A)}\mathrm{Hom}_{\mathrm{H}^0(A)}(\mathrm{H}^{\mathrm{sup}R}(R),\mathrm{H}^{\mathrm{inf}(M\otimes_A^\mathrm{L}R)}(M\otimes_A^\mathrm{L}R))=
\mathrm{Supp}_{\mathrm{H}^0(A)}\mathrm{H}^{\mathrm{sup}R}(R)\cap\mathrm{Ass}_{\mathrm{H}^0(A)}\mathrm{H}^{\mathrm{inf}(M\otimes_A^\mathrm{L}R)}(M\otimes_A^\mathrm{L}R)\neq\emptyset$.\end{center}
Hence
 $\mathrm{H}^{\mathrm{inf}(M\otimes_A^\mathrm{L}R)-\mathrm{sup}R}(\mathrm{RHom}_A(R,M\otimes_A^\mathrm{L}R))=
\mathrm{Hom}_{\mathrm{H}^0(A)}(\mathrm{H}^{\mathrm{sup}R}(R),\mathrm{H}^{\mathrm{inf}(M\otimes_A^\mathrm{L}R)}(M\otimes_A^\mathrm{L}R))\neq0$ by \cite[Lemma 3.2]{BSSW}. Note that $M\simeq\mathrm{RHom}_A(R,M\otimes_A^\mathrm{L}R)$ by \cite[Lemma 2.7]{BSSW}, so $\mathrm{inf}(M\otimes_A^\mathrm{L}R)=\mathrm{inf}M+\mathrm{sup}R$.
 As $\mathrm{RHom}_A(\mathrm{RHom}_A(M,A),R)\simeq M\otimes_A^\mathrm{L}R$ by  \cite[Lemma 2.7]{BSSW}, it follows by \cite[Theorem 7.26]{s18} and  \cite[Lemma 2.3(3)]{Y25} that $$\mathrm{amp}\mathrm{R}\Gamma_{\bar{\mathfrak{m}}}(\mathrm{RHom}_A(M,A))=\mathrm{amp}\mathrm{RHom}_A(\mathrm{RHom}_A(M,A),R)=\mathrm{amp}(M\otimes_A^\mathrm{L}R)=\mathrm{amp}M.$$

 (1) $\Rightarrow$ (3) As $M$ is local Cohen-Macaulay, $M$ is perfect by Remark \ref{lem:2.2}(4), and so $\mathrm{amp}A=\mathrm{ampRHom}_A(M,A)$ by Remark \ref{lem:2.2}(2). Thus $\mathrm{RHom}_A(M,A)$ is local Cohen-Macaulay.

(3) $\Rightarrow$ (1) As $\mathrm{RHom}_A(M,A)$ is local Cohen-Macaulay,  $\mathrm{amp}A=\mathrm{ampRHom}_A(M,A)=\mathrm{amp}\mathrm{R}\Gamma_{\bar{\mathfrak{m}}}(\mathrm{RHom}_A(M,A))=\mathrm{amp}M$. Also $\mathrm{amp}A=\mathrm{projdim}_AM-\mathrm{grade}_AM=\mathrm{amp}\mathrm{R}\Gamma_{\bar{\mathfrak{m}}}(M)$, it implies that $M$ is local Cohen-Macaulay.
\end{proof}

The next proposition provides a partial answer to the Codimension Conjecture proposed by Auslander.

\begin{prop}\label{lem3.1} Let $A$ be a local DG-ring  with constant amplitude and $M\in\mathrm{D}^{\mathrm{b}}_{\mathrm{f}}(A)$ with $\mathrm{projdim}_AM<\infty$. Then
$M$ is perfect if and only if $\mathrm{lc.dim}_AM+\mathrm{grade}_AM=\mathrm{lc.dim}A+\mathrm{inf}A$ and $\mathrm{ampR}\Gamma_{\bar{\mathfrak{m}}}(M)=\mathrm{ampR}\Gamma_{\bar{\mathfrak{m}}}(A)$.
\end{prop}
\begin{proof} One has the following inequalities
\begin{center}$\begin{aligned}\mathrm{lc.dim}_AM+\mathrm{grade}_AM
&\leq\mathrm{lc.dim}A+\mathrm{inf}A\\
&\leq\mathrm{lc.dim}_AM-\mathrm{depth}_AM+\mathrm{depth}A+\mathrm{inf}A\\
&=\mathrm{lc.dim}_AM+\mathrm{projdim}_AM+\mathrm{inf}A,\end{aligned}$\end{center}
where the first inequality is by Lemma \ref{lem2.4}(2) and the second one is by \cite[Corollary 3.3]{Y25}, the last equality is by Lemma \ref{lem0.15}(3), as desired.
\end{proof}

 The affirmative answer to the conjecture of Yoshida in the introduction is the special case of the following theorem.

\begin{thm}\label{lem3.4} Let $A$ be a local DG-ring with constant amplitude and $M,N\in\mathrm{D}^{\mathrm{b}}_{\mathrm{f}}(A)$ such that $M$ is perfect, $\mathrm{amp}N=\mathrm{amp}A$ and $\mathrm{lc.dim}_AN-\mathrm{sup}N=\mathrm{lc.dim}A$. If $M\otimes^\mathrm{L}_AN$ is local  Cohen-Macaulay, then $\mathrm{lc.dim}_A(M\otimes^\mathrm{L}_AN)=\mathrm{lc.dim}_AM+\mathrm{sup}N$ and $N$ is a maximal local Cohen-Macaulay DG-module.
\end{thm}
\begin{proof} For $\bar{\mathfrak{p}}\in\mathrm{Spec}\mathrm{H}^0(A)$, one has $\mathrm{dim}\mathrm{H}^0(A)/\bar{\mathfrak{p}}+\mathrm{sup}(M\otimes^\mathrm{L}_AN)_{\bar{\mathfrak{p}}}=
\mathrm{dim}\mathrm{H}^0(A)/\bar{\mathfrak{p}}+\mathrm{sup}M_{\bar{\mathfrak{p}}}+\mathrm{sup}N_{\bar{\mathfrak{p}}}\leq\mathrm{lc.dim}_AM+\mathrm{sup}N$ by \cite[Lemma 2.3(3)]{Y25}, it implies that
$\mathrm{lc.dim}_A(M\otimes^\mathrm{L}_AN)\leq\mathrm{lc.dim}_AM+\mathrm{sup}N$. On the other hand, $-\mathrm{inf}A=\mathrm{ampR}\Gamma_{\bar{\mathfrak{m}}}(M\otimes^\mathrm{L}_AN)\geq\mathrm{ampR}\Gamma_{\bar{\mathfrak{m}}}(N)$ by \cite[Corollary 3.3]{Y25}, it yields the following  (in)equalities:
 \begin{center}$\begin{aligned}\mathrm{lc.dim}_A(M\otimes^\mathrm{L}_AN)
 &=\mathrm{depth}_A(M\otimes^\mathrm{L}_AN)-\mathrm{inf}A\\
&=\mathrm{depth}_AN-\mathrm{inf}A-\mathrm{projdim}_AM\\
&\geq\mathrm{lc.dim}_AN+\mathrm{inf}A-\mathrm{grade}_AM\\
&\geq\mathrm{lc.dim}_AN+\mathrm{lc.dim}_AM-\mathrm{lc.dim}A\\
&=\mathrm{lc.dim}_AM+\mathrm{sup}N,\end{aligned}$\end{center}
where the second equality is by Lemma \ref{lem0.15} (1) and (3), the inequalities are by $\mathrm{ampR}\Gamma_{\bar{\mathfrak{m}}}(N)\leq-\mathrm{inf}A$ and
 Lemma \ref{lem2.4}(2). So $\mathrm{lc.dim}_A(M\otimes^\mathrm{L}_AN)=\mathrm{lc.dim}_AM+\mathrm{sup}N$. Also
\begin{center}$\begin{aligned}\mathrm{lc.dim}_AN-\mathrm{depth}_AN
&=\mathrm{lc.dim}A+\mathrm{sup}N+\mathrm{depth}_AM-\mathrm{depth}_A(M\otimes^\mathrm{L}_AN)-\mathrm{depth}A\\
&=\mathrm{lc.dim}A+\mathrm{sup}N-\mathrm{projdim}_AM-\mathrm{lc.dim}_A(M\otimes^\mathrm{L}_AN)-\mathrm{inf}A\\
&=\mathrm{lc.dim}A+\mathrm{sup}N-\mathrm{grade}_AM-\mathrm{lc.dim}_AM-\mathrm{sup}N\\
&\geq-\mathrm{inf}A,\end{aligned}$\end{center}where the first equality is by Lemma \ref{lem0.15}(1) and the inequality is
by Lemma \ref{lem2.4}(2). Hence $\mathrm{ampR}\Gamma_{\bar{\mathfrak{m}}}(N)=\mathrm{amp}A$ and $N$ is  maximal local Cohen-Macaulay.
\end{proof}

We now consider the converse of Theorem \ref{lem3.4}.

\begin{prop}\label{lem3.0} Let $A$ be a local DG-ring with constant amplitude and $M,N\in\mathrm{D}^{\mathrm{b}}_{\mathrm{f}}(A)$ such that $M$ is perfect and $N$ is maximal local Cohen-Macaulay.

$(1)$ $\mathrm{lc.dim}_A(M\otimes^\mathrm{L}_AN)=\mathrm{lc.dim}_AM+\mathrm{lc.dim}_AN-\mathrm{lc.dim}A$.

$(2)$ If $\mathrm{ampR}\Gamma_{\bar{\mathfrak{m}}}(M\otimes^\mathrm{L}_AN)\geq\mathrm{amp}(M\otimes^\mathrm{L}_AN)$, then $M\otimes^\mathrm{L}_AN$ is local  Cohen-Macaulay.
\end{prop}
\begin{proof} (1) One has the following inequalities:
\begin{center}$\begin{aligned}\mathrm{lc.dim}_A(M\otimes^\mathrm{L}_AN)
&\geq\mathrm{lc.dim}_AN-\mathrm{projdim}_AM\\
&=\mathrm{lc.dim}A+\mathrm{sup}N-\mathrm{grade}_AM+\mathrm{inf}A\\
&\geq\mathrm{lc.dim}A+\mathrm{sup}N+\mathrm{lc.dim}_AM-\mathrm{lc.dim}A\\
&=\mathrm{lc.dim}_AM+\mathrm{sup}N,\end{aligned}$\end{center}
where the first inequality is by Lemma \ref{lem0.15}(2), the second inequality is by Lemma \ref{lem2.4}(2). On the other hand, $\mathrm{lc.dim}_A(M\otimes^\mathrm{L}_AN)\leq\mathrm{lc.dim}_AM+\mathrm{sup}N$
by the proof of Theorem \ref{lem3.4}. Hence $\mathrm{lc.dim}_A(M\otimes^\mathrm{L}_AN)=\mathrm{lc.dim}_AM+\mathrm{lc.dim}_AN-\mathrm{lc.dim}A$.

(2) As $N$ is maximal local Cohen-Macaulay, $\mathrm{depth}_AN=\mathrm{lc.dim}_AN-\mathrm{amp}N=\mathrm{lc.dim}A+\mathrm{inf}N$. One has the following (in)equalities:
 \begin{center}$\begin{aligned}\mathrm{ampR}\Gamma_{\bar{\mathfrak{m}}}(M\otimes^\mathrm{L}_AN)
 &=\mathrm{lc.dim}_A(M\otimes^\mathrm{L}_AN)-\mathrm{depth}_A(M\otimes^\mathrm{L}_AN)\\
 &=\mathrm{lc.dim}_AM+\mathrm{projdim}_AM+\mathrm{sup}N-\mathrm{depth}_AN\\
&=\mathrm{lc.dim}_AM+\mathrm{grade}_AM+\mathrm{sup}N-\mathrm{depth}_AN-\mathrm{inf}A\\
&=\mathrm{lc.dim}_AM+\mathrm{grade}_AM+\mathrm{amp}N-\mathrm{inf}A-\mathrm{lc.dim}A\\
&\leq\mathrm{lc.dim}_AM+\mathrm{lc.dim}A-\mathrm{lc.dim}_AM+\mathrm{inf}A+\mathrm{amp}N-\mathrm{inf}A-\mathrm{lc.dim}A\\
&=\mathrm{amp}N,\end{aligned}$\end{center}
 where the second equality is by Lemma \ref{lem0.15} (1) and (3), the inequality is
 by Lemma \ref{lem2.4}(2). On the other hand, $\mathrm{amp}(M\otimes^\mathrm{L}_AN)\geq\mathrm{amp}N$ by \cite[Proposition 4.1]{Y25},
 it follows by assumption that $M\otimes^\mathrm{L}_AN$ is local  Cohen-Macaulay.
\end{proof}

\begin{cor}\label{lem3.2} Let $A$ be a local  Cohen-Macaulay DG-ring with constant amplitude and $M\in\mathrm{D}^{\mathrm{b}}_{\mathrm{f}}(A)$ with $\mathrm{projdim}_AM<\infty$ and $\mathrm{ampR}\Gamma_{\bar{\mathfrak{m}}}(M)\geq\mathrm{amp}M$. The following are equivalent:

$(1)$ $M$ is local Cohen-Macaulay;

$(2)$ $M\otimes^\mathrm{L}_AN$ is local  Cohen-Macaulay for any  maximal local Cohen-Macaulay DG-module $N$ with $\mathrm{ampR}\Gamma_{\bar{\mathfrak{m}}}(M\otimes^\mathrm{L}_AN)\geq\mathrm{amp}(M\otimes^\mathrm{L}_AN)$;

$(3)$ $M\otimes^\mathrm{L}_AN$ is local  Cohen-Macaulay for some  maximal local Cohen-Macaulay DG-module $N$ with $\mathrm{ampR}\Gamma_{\bar{\mathfrak{m}}}(M\otimes^\mathrm{L}_AN)\geq\mathrm{amp}(M\otimes^\mathrm{L}_AN)$;

$(4)$ $M\otimes^\mathrm{L}_AR$ is local  Cohen-Macaulay and $\mathrm{ampR}\Gamma_{\bar{\mathfrak{m}}}(M\otimes^\mathrm{L}_AR)\geq\mathrm{amp}M$, where $R$ is a normalized dualizing DG-module for $A$.
\end{cor}
\begin{proof} (1) $\Rightarrow$ (2)  This follows by Remark \ref{lem:2.2}(4) and Proposition \ref{lem3.0}(2).

(2) $\Rightarrow$ (3) is trivial. (3) $\Rightarrow$ (1) Set $N=A$.

(4) $\Rightarrow$ (3) is trivial as $R$ is maximal local Cohen-Macaulay by \cite[Theorem 6.7]{Shaul}
and $\mathrm{amp}(M\otimes^\mathrm{L}_AR)=\mathrm{amp}M$ by the proof of Theorem \ref{lem2.5}.

(1) $\Rightarrow$ (4) By \cite[Theorem 7.26]{s18}, Remark \ref{lem:2.2} (4) and (2), one has $\mathrm{ampR}\Gamma_{\bar{\mathfrak{m}}}(M\otimes^\mathrm{L}_AR)=\mathrm{ampRHom}_A(M\otimes^\mathrm{L}_AR,R)=\mathrm{ampRHom}_A(M,A)=\mathrm{amp}A$ and $\mathrm{amp}(M\otimes^\mathrm{L}_AR)=\mathrm{amp}M$ by the proof of Theorem \ref{lem2.5}, as claimed.
\end{proof}

\begin{prop}\label{lem2.9} Let $A$ be a local DG-ring and $M\in\mathrm{D}^{\mathrm{b}}_{\mathrm{f}}(A)$ be perfect with constant amplitude. The following are equivalent:

$(1)$ $\bar{\mathfrak{p}}\in\mathrm{Ass}_AM$;

$(2)$ $\mathrm{grade}_AM=\mathrm{depth}A_{\bar{\mathfrak{p}}}-\mathrm{sup}M_{\bar{\mathfrak{p}}}$ and $\mathrm{projdim}A_{\bar{\mathfrak{p}}}M_{\bar{\mathfrak{p}}}=\mathrm{projdim}_AM$.
\end{prop}
\begin{proof} As $M$ is perfect, $\mathrm{projdim}_AM<\infty$, so $\mathrm{amp}M\geq-\mathrm{inf}A$ by \cite[Proposition 4.1]{Y25}. As $M$ has constant amplitude, $\mathrm{inf}M_{\bar{\mathfrak{p}}}=\mathrm{inf}M$ and $\mathrm{sup}M_{\bar{\mathfrak{p}}}=\mathrm{sup}M$.
One has inequalities $\mathrm{grade}_AM\leq\mathrm{depth}A_{\bar{\mathfrak{p}}}-\mathrm{sup}M_{\bar{\mathfrak{p}}}=
\mathrm{projdim}_{A_{\bar{\mathfrak{p}}}}M_{\bar{\mathfrak{p}}}+\mathrm{depth}_{A_{\bar{\mathfrak{p}}}}M_{\bar{\mathfrak{p}}}-\mathrm{sup}M_{\bar{\mathfrak{p}}}
\leq\mathrm{projdim}_AM+\mathrm{depth}_{A_{\bar{\mathfrak{p}}}}M_{\bar{\mathfrak{p}}}-\mathrm{inf}M-\mathrm{amp}M\leq
\mathrm{projdim}_AM+\mathrm{depth}_{A_{\bar{\mathfrak{p}}}}M_{\bar{\mathfrak{p}}}-\mathrm{inf}M+\mathrm{inf}A$.   Thus $\bar{\mathfrak{p}}\in\mathrm{Ass}_AM$ if and only if $\mathrm{depth}_{A_{\bar{\mathfrak{p}}}}M_{\bar{\mathfrak{p}}}-\mathrm{inf}M_{\bar{\mathfrak{p}}}=0$ if and only if $\mathrm{depth}A_{\bar{\mathfrak{p}}}-\mathrm{sup}M_{\bar{\mathfrak{p}}}=\mathrm{grade}_AM$ and $\mathrm{projdim}_{A_{\bar{\mathfrak{p}}}}M_{\bar{\mathfrak{p}}}=\mathrm{projdim}_AM$.
\end{proof}

The next corollary provides criteria for the DG-ring $A$ to be Cohen-Macaulay.

\begin{cor}\label{lem2.10} Let $A$ be a local DG-ring  with constant amplitude. The next are equivalent:

$(1)$ $A$ is local Cohen-Macaulay;

$(2)$ There exists a local Cohen-Macaulay DG-module $M\not\simeq0$ with $\mathrm{projdim}_AM<\infty$ and
$\mathrm{lc.dim}_AM=\mathrm{lc.dim}A-\mathrm{grade}_AM+\mathrm{inf}A$;

$(3)$ There exists a local Cohen-Macaulay DG-module $M\not\simeq0$ with $\mathrm{projdim}_AM<\infty$ and
$\mathrm{lc.dim}A=\mathrm{dim}\mathrm{H}^0(A)/\bar{\mathfrak{p}}+\mathrm{depth}A_{\bar{\mathfrak{p}}}-\mathrm{inf}A$ for some $\bar{\mathfrak{p}}\in\mathrm{Ass}_AM$.
\end{cor}
\begin{proof} The assertions (1) $\Rightarrow$ (2) and (1) $\Rightarrow$ (3) are trivial by setting $M=A$.

(2) $\Rightarrow$ (1) By Remark \ref{lem:2.2}(4), $M$ is perfect, one has $\mathrm{lc.dim}A-\mathrm{projdim}_AM=\mathrm{lc.dim}A-\mathrm{grade}_AM+\mathrm{inf}A=\mathrm{lc.dim}_AM=\mathrm{depth}_AM-\mathrm{inf}A=\mathrm{depth}A-\mathrm{projdim}_AM-\mathrm{inf}A$ by Lemma \ref{lem0.15}(3). Hence $A$ is local Cohen-Macaulay.

(3) $\Rightarrow$ (1) There exists $\bar{\mathfrak{p}}\in\mathrm{Ass}_AM$ such that $\mathrm{lc.dim}_AM=\mathrm{dim}\mathrm{H}^0(A)/\bar{\mathfrak{p}}+\mathrm{sup}M_{\bar{\mathfrak{p}}}$. Then  $\mathrm{depth}A_{\bar{\mathfrak{p}}}-\mathrm{sup}M_{\bar{\mathfrak{p}}}=\mathrm{grade}_AM$ by Proposition \ref{lem2.9}. As $M$ is perfect by Remark \ref{lem:2.2}(4), it implies that $\mathrm{lc.dim}A+\mathrm{inf}A=\mathrm{dim}\mathrm{H}^0(A)/\bar{\mathfrak{p}}+\mathrm{depth}A_{\bar{\mathfrak{p}}}=\mathrm{lc.dim}_AM+\mathrm{grade}_AM
=\mathrm{depth}_AM+\mathrm{projdim}_AM=\mathrm{depth}A$, as claimed.
\end{proof}

\bigskip
\section{\bf Grade of DG-modules}
This section, we study the grade of DG-modules of finite injective dimension. Using the Cohen-Macaulayness of the DG-module $M/\hspace{-0.15cm}/\bar{\emph{\textbf{x}}}$, we explore the preservation of Cohen-Macaulayness under tensor products in the DG context.

\begin{lem}\label{lem3.9} Let $A$ be a local DG-ring with constant amplitude, $\bar{\textbf{x}}=\bar{x}_1,\cdots,\bar{x}_n\in\bar{\mathfrak{m}}$ a finite sequence and $M\in\mathrm{D}^{\mathrm{b}}_{\mathrm{f}}(A)$. Then

$(1)$  $\mathrm{grade}_{A/\hspace{-0.1cm}/\bar{x}}M/\hspace{-0.15cm}/\bar{\textbf{x}}=\mathrm{grade}_AM/\hspace{-0.15cm}/\bar{x}-n$.

$(2)$  $\mathrm{grade}_{A}M\leq\mathrm{grade}_AM/\hspace{-0.15cm}/\bar{x}\leq\mathrm{grade}_{A}M+n$.

$(3)$ If $\bar{\textbf{x}}$ is $A$-regular, then $\mathrm{grade}_AM/\hspace{-0.15cm}/\bar{\textbf{x}}=\mathrm{grade}_AM$.
\end{lem}
\begin{proof} As $M/\hspace{-0.15cm}/\bar{\emph{\textbf{x}}}\simeq A/\hspace{-0.15cm}/\bar{\emph{\textbf{x}}}\otimes^\mathrm{L}_AM$, by \cite[Proposition 2.12]{sh20}, one has isomorphisms in $\mathrm{D}(A)$
$$\mathrm{RHom}_{A/\hspace{-0.1cm}/\bar{\emph{\textbf{x}}}}(M/\hspace{-0.15cm}/\bar{\emph{\textbf{x}}},A/\hspace{-0.15cm}/\bar{\emph{\textbf{x}}})
\simeq\mathrm{RHom}_A(M,A/\hspace{-0.15cm}/\bar{\emph{\textbf{x}}}),$$ $$\mathrm{RHom}_A(M/\hspace{-0.15cm}/\bar{\emph{\textbf{x}}},A)\simeq\mathrm{RHom}_A(M,\mathrm{RHom}_A(A/\hspace{-0.15cm}/\bar{\emph{\textbf{x}}},A))
\simeq\mathrm{RHom}_A(M,A/\hspace{-0.15cm}/\bar{\emph{\textbf{x}}})[-n].$$

(1) This follows by the above two isomorphisms.

(2) One has the following inequalities:\begin{center}$\begin{aligned}\mathrm{grade}_AM
&\leq\mathrm{inf}\mathrm{RHom}_A(M,A)-\mathrm{projdim}_AA/\hspace{-0.15cm}/\bar{\emph{\textbf{x}}}+n\\
&\leq\mathrm{inf}(\mathrm{RHom}_A(M,A)\otimes^\mathrm{L}_AA/\hspace{-0.15cm}/\bar{\emph{\textbf{x}}})+n\\
&=\mathrm{inf}\mathrm{RHom}_A(M,A/\hspace{-0.15cm}/\bar{\emph{\textbf{x}}})+n\\
&=\mathrm{inf}\mathrm{RHom}_A(M/\hspace{-0.15cm}/\bar{\emph{\textbf{x}}},A)\\
&=\mathrm{grade}_AM/\hspace{-0.15cm}/\bar{\emph{\textbf{x}}},\end{aligned}$\end{center}
where the first inequality is by $\mathrm{projdim}_AA/\hspace{-0.15cm}/\bar{\emph{\textbf{x}}}\leq n$ and the second one is by \cite[Proposition 3.14]{ya20}, the second equality is by the above isomorphisms. Also by \cite[Lemma 2.13(1)]{Mi19}, $$\mathrm{grade}_AM/\hspace{-0.15cm}/\bar{\emph{\textbf{x}}}=\mathrm{inf}(\mathrm{RHom}_A(M,A)\otimes^\mathrm{L}_AA/\hspace{-0.15cm}/\bar{\emph{\textbf{x}}})+n\leq\mathrm{inf}\mathrm{RHom}_A(M,A)+n
=\mathrm{grade}_{A}M+n.$$

(3) As $\mathrm{Spec}\mathrm{H}^0(A)=\mathrm{Supp}_{\mathrm{H}^0(A)}\mathrm{H}^{\mathrm{inf}A}(A)$, $\bar{x}_1/1,\cdots,\bar{x}_n/1$ is $A_{\bar{\mathfrak{p}}}$-regular for any $\hspace{0.02cm}\bar{\mathfrak{p}}\in\mathrm{Supp}_AM\cap\mathrm{Supp}_AA/\hspace{-0.15cm}/\bar{\emph{\textbf{x}}}$.
By the second isomorphism, one has
\begin{center}$\begin{aligned}\mathrm{grade}_AM/\hspace{-0.15cm}/\bar{\emph{\textbf{x}}}
&=\mathrm{inf}\mathrm{RHom}_A(M,A/\hspace{-0.15cm}/\bar{\emph{\textbf{x}}})+n\\
&=\mathrm{inf}\{\mathrm{depth}(A/\hspace{-0.15cm}/\bar{\emph{\textbf{x}}})_{\bar{\mathfrak{p}}}
-\mathrm{sup}M_{\bar{\mathfrak{p}}}\hspace{0.02cm}|\hspace{0.02cm}\bar{\mathfrak{p}}\in\mathrm{Spec}\mathrm{H}^0(A)\}+n\\
&=\mathrm{inf}\{\mathrm{depth}A_{\bar{\mathfrak{p}}}
-\mathrm{sup}M_{\bar{\mathfrak{p}}}\hspace{0.02cm}-n|\hspace{0.02cm}\bar{\mathfrak{p}}\in\mathrm{Spec}\mathrm{H}^0(A)\}+n\\
&=\mathrm{grade}_AM,\end{aligned}$\end{center}where the second is by the proof of \cite[Lemma 2.3(1)]{Y25}, the third one is by \cite[Lemma 3.6]{YL} and \cite[Lemma 2.13(2)]{Mi19}.
\end{proof}

\begin{rem}\label{lem:2.3} {\rm Let $A$ be a local DG-ring and $M\in\mathrm{D}^{\mathrm{b}}_{\mathrm{f}}(A)$ with $\mathrm{injdim}_AM<\infty$. By \cite[Remark 7.17]{s18}, $\mathrm{injdim}_A\mathrm{R}\Gamma_{\bar{\mathfrak{m}}}(M)\leq\mathrm{injdim}_AM$, so $\mathrm{supR}\Gamma_{\bar{\mathfrak{m}}}(M)\leq\mathrm{sup}M$ by Lemma \ref{lem0.15}(4). Hence \cite[Lemma 2.1(2)]{Y25} implies that $$\mathrm{lc.dim}_AM\leq\mathrm{sup}M\leq\mathrm{injdim}_AM-\mathrm{inf}A
=\mathrm{depth}A+\mathrm{sup}M-\mathrm{inf}A.$$}
\end{rem}

The next result improves the above inequality, it is also a dual version of Lemma \ref{lem0.15}(3).

\begin{thm}\label{lem2.6} Let $A$ be a local DG-ring with constant amplitude and $M\in\mathrm{D}^{\mathrm{b}}_{\mathrm{f}}(A)$ with constant amplitude and $\mathrm{injdim}_AM<\infty$.
Then $$\mathrm{lc.dim}_AM=\mathrm{depth}A-\mathrm{grade}_AM.$$
\end{thm}
\begin{proof} As $\mathrm{injdim}_AM<\infty$, $\mathrm{injdim}_{A_{\bar{\mathfrak{q}}}}M_{\bar{\mathfrak{q}}}<\infty$ for $\bar{\mathfrak{q}}\in\mathrm{Spec}\mathrm{H}^0(A)$. By the equality $(\ast)$, choose $\bar{\mathfrak{p}}\in\mathrm{Spec}\mathrm{H}^0(A)$ such that $\mathrm{grade}_AM=\mathrm{depth}A_{\bar{\mathfrak{p}}}
-\mathrm{sup}M_{\bar{\mathfrak{p}}}$. Then $\mathrm{grade}_AM=\mathrm{injdim}_{A_{\bar{\mathfrak{p}}}}M_{\bar{\mathfrak{p}}}-2\mathrm{sup}M_{\bar{\mathfrak{p}}}<\infty$ by Lemma \ref{lem0.15}(4).
We proceed by induction on $\mathrm{grade}_AM\geq\mathrm{inf}A-\mathrm{sup}M$. If $\mathrm{grade}_AM=\mathrm{inf}A-\mathrm{sup}M$, then
$\mathrm{depth}A\leq\mathrm{lc.dim}_AM-\mathrm{sup}M+\mathrm{inf}A$ by Lemma \ref{lem2.4}(1). On the other hand, for $\bar{\mathfrak{p}}\in\mathrm{Spec}\mathrm{H}^0(A)$, one has the following inequalities
\begin{center}$\begin{aligned}\mathrm{dim}\mathrm{H}^0(A)/\bar{\mathfrak{p}}+\mathrm{sup}M_{\bar{\mathfrak{p}}}
&\leq\mathrm{dim}\mathrm{H}^0(A)/\bar{\mathfrak{p}}+\mathrm{injdim}_{A_{\bar{\mathfrak{p}}}}M_{\bar{\mathfrak{p}}}-\mathrm{inf}A_{\bar{\mathfrak{p}}}\\
&\leq\mathrm{injdim}_{A}M-\mathrm{inf}A,\end{aligned}$\end{center}
where the first one is by \cite[Lemma 2.1(2)]{Y25}, the second one is by \cite[Lemma 2.26 and Proposition 2.29]{Mi19}.
Thus $\mathrm{lc.dim}_AM-\mathrm{sup}M+\mathrm{inf}A\leq\mathrm{injdim}_{A}M-\mathrm{sup}M=\mathrm{depth}A$ by \cite[Equality (2.14)]{Shaul} and Lemma \ref{lem0.15}(4), and the desired equality holds. Now assume that $\mathrm{grade}_AM>\mathrm{inf}A-\mathrm{sup}M$. Then $\mathrm{Hom}_{\mathrm{H}^0(A)}(\mathrm{H}^{\mathrm{sup}M}(M),\mathrm{H}^{\mathrm{inf}A}(A))=0$, it follows by \cite[Theorem 16.6]{M} that there is an $A$-regular element $\bar{x}\in\bar{\mathfrak{m}}$ such that $\bar{x}\mathrm{H}^{\mathrm{sup}M}(M)=0$. As $\mathrm{inf}A/\hspace{-0.15cm}/\bar{x}-\mathrm{sup}M/\hspace{-0.15cm}/\bar{x}=\mathrm{inf}A-\mathrm{sup}M$,  $\mathrm{injdim}_{A/\hspace{-0.1cm}/\bar{x}}M/\hspace{-0.15cm}/\bar{x}<\infty$ and $\mathrm{grade}_{A/\hspace{-0.1cm}/\bar{x}}M/\hspace{-0.15cm}/\bar{x}=\mathrm{grade}_AM-1$ by Lemma \ref{lem3.9}, it follows by the induction that $$\mathrm{lc.dim}_{A/\hspace{-0.1cm}/\bar{x}}M/\hspace{-0.15cm}/\bar{x}=
 \mathrm{depth}A/\hspace{-0.15cm}/\bar{x}-\mathrm{grade}_{A/\hspace{-0.1cm}/\bar{x}}M/\hspace{-0.15cm}/\bar{x}=\mathrm{depth}A-\mathrm{grade}_AM.$$
Consider the exact triangle $M\stackrel{\bar{x}}\rightarrow M\rightarrow M/\hspace{-0.15cm}/\bar{x}\rightarrow M[1]$ in $\mathrm{D}(A)$. For $\ell\in\mathbb{Z}$, one has the following short exact sequence $$0\rightarrow\mathrm{H}^\ell(M)/\bar{x}\mathrm{H}^\ell(M)\rightarrow\mathrm{H}^\ell(M/\hspace{-0.15cm}/\bar{x})\rightarrow
(0:_{\mathrm{H}^{\ell+1}(M)}\bar{x})\rightarrow0,$$it follows by \cite[Proposition 14.2.6]{CFH} that
\begin{center}$\begin{aligned}\mathrm{dim}_{\mathrm{H}^0(A)/(\bar{x})}\mathrm{H}^\ell(M/\hspace{-0.15cm}/\bar{x})
&=\mathrm{max}\{\mathrm{dim}_{\mathrm{H}^0(A)/(\bar{x})}\mathrm{H}^\ell(M)/\bar{x}\mathrm{H}^\ell(M),
\mathrm{dim}_{\mathrm{H}^0(A)/(\bar{x})}(0:_{\mathrm{H}^{\ell+1}(M)}\bar{x})\}\\
&=\mathrm{max}\{\mathrm{dim}_{\mathrm{H}^0(A)}\mathrm{H}^\ell(M)/\bar{x}\mathrm{H}^\ell(M),
\mathrm{dim}_{\mathrm{H}^0(A)}(0:_{\mathrm{H}^{\ell+1}(M)}\bar{x})\}\\
&=\mathrm{dim}_{\mathrm{H}^0(A)}\mathrm{H}^\ell(M/\hspace{-0.15cm}/\bar{x}).\end{aligned}$\end{center}
As $\mathrm{Supp}_AM=\mathrm{Supp}_{\mathrm{H}^{0}(A)}\mathrm{H}^{\mathrm{sup}M}(M)$ and $\bar{x}\mathrm{H}^{\mathrm{sup}M}(M)=0$, it follows that $\mathrm{Supp}_AM=\mathrm{Supp}_AM/\hspace{-0.15cm}/\bar{x}$. Note that $\mathrm{H}^0(A/\hspace{-0.15cm}/\bar{x})\cong\mathrm{H}^0(A)/(\bar{x})$, one has
 \begin{center}$\begin{aligned}\mathrm{lc.dim}_{A/\hspace{-0.1cm}/\bar{x}}M/\hspace{-0.15cm}/\bar{x}
&=\mathrm{sup}\{\mathrm{dim}_{\mathrm{H}^0(A)/(\bar{x})}\mathrm{H}^\ell(M/\hspace{-0.15cm}/\bar{x})+\ell\hspace{0.02cm}|\hspace{0.02cm}\ell\in\mathbb{Z}\}\\
&=\mathrm{sup}\{\mathrm{dim}_{\mathrm{H}^0(A)}\mathrm{H}^\ell(M/\hspace{-0.15cm}/\bar{x})+\ell\hspace{0.02cm}|\hspace{0.02cm}\ell\in\mathbb{Z}\}\\
&=\mathrm{lc.dim}_{A}M/\hspace{-0.15cm}/\bar{x}\\
&=\mathrm{sup}\{\mathrm{dim}\mathrm{H}^0(A)/\bar{\mathfrak{p}}+
\mathrm{sup}(M/\hspace{-0.15cm}/\bar{x})_{\bar{\mathfrak{p}}}\hspace{0.02cm}|\hspace{0.02cm}\bar{\mathfrak{p}}\in\mathrm{Supp}_AM\}\\
&=\mathrm{sup}\{\mathrm{dim}\mathrm{H}^0(A)/\bar{\mathfrak{p}}+
\mathrm{sup}M_{\bar{\mathfrak{p}}}\hspace{0.02cm}|\hspace{0.02cm}\bar{\mathfrak{p}}\in\mathrm{Supp}_AM\}\\
&=\mathrm{lc.dim}_{A}M,\end{aligned}$\end{center}
  where the fourth one is by \cite[Equality (2.14)]{Shaul}.
Therefore, $\mathrm{lc.dim}_AM=\mathrm{lc.dim}_{A/\hspace{-0.1cm}/\bar{x}}M/\hspace{-0.15cm}/\bar{x}=\mathrm{depth}A-\mathrm{grade}_AM$.
\end{proof}

The next result is a dual version of Theorem \ref{lem2.5}.

\begin{cor}\label{lem2.8} Let $A$ be a local DG-ring with constant amplitude and $0\not\simeq M\in\mathrm{D}^{\mathrm{b}}_{\mathrm{f}}(A)$ with constant amplitude and $\mathrm{injdim}_AM<\infty$. Then $M$ is perfect with $\mathrm{amp}M\leq\mathrm{amp}\mathrm{R}\Gamma_{\bar{\mathfrak{m}}}(M)$ if and only if $M$ is local Cohen-Macaulay.
\end{cor}
\begin{proof} Let $M$ be perfect with $\mathrm{amp}M\leq\mathrm{amp}\mathrm{R}\Gamma_{\bar{\mathfrak{m}}}(M)$. As $\mathrm{injdim}_AM<\infty$, it follows by \cite[Proposition 4.5]{Y25} and \cite[Corollary 5.5]{Shaul} that $\mathrm{amp}M\geq\mathrm{lc.dim}A-\mathrm{depth}A\geq-\mathrm{inf}A$. By Theorem \ref{lem2.6}, $\mathrm{lc.dim}_AM=\mathrm{depth}A-\mathrm{grade}_AM=\mathrm{depth}A-\mathrm{projdim}_AM-\mathrm{inf}A=\mathrm{depth}_AM-\mathrm{inf}A$. Thus $\mathrm{amp}M\leq\mathrm{amp}\mathrm{R}\Gamma_{\bar{\mathfrak{m}}}(M)=-\mathrm{inf}A\leq\mathrm{amp}M$ and so $M$ is local Cohen-Macaulay. The converse follows by Remark \ref{lem:2.2}(4).
\end{proof}

The next theorem provides a condition for $M/\hspace{-0.15cm}/\bar{\emph{\textbf{x}}}$ to be perfect and Cohen-Macaulay.

\begin{thm}\label{lem2.3} Let $(A,\bar{\mathfrak{m}})$ be a local DG-ring and $\bar{\textbf{x}}=\bar{x}_1,\cdots,\bar{x}_n\in\bar{\mathfrak{m}}$ a finite sequence, and let $M\in\mathrm{D}^{\mathrm{b}}_{\mathrm{f}}(A)$.

$(1)$ If $A$ has constant amplitude and $M/\hspace{-0.15cm}/\bar{\textbf{x}}$ is perfect, then $\mathrm{lc.dim}_AM/\hspace{-0.15cm}/\bar{\textbf{x}}=\mathrm{lc.dim}_AM-n$.

$(2)$  If $A$ is local Cohen-Macaulay with constant amplitude, then $M/\hspace{-0.15cm}/\bar{\textbf{x}}$ is perfect if and only if $M$ is perfect and $\mathrm{lc.dim}_AM/\hspace{-0.15cm}/\bar{\textbf{x}}=\mathrm{lc.dim}_AM-n$.

$(3)$ If $M$ is local Cohen-Macaulay with constant amplitude, then  $\mathrm{amp}M/\hspace{-0.15cm}/\bar{\textbf{x}}=\mathrm{ampR}\Gamma_{\mathfrak{\bar{m}}}(M/\hspace{-0.15cm}/\bar{\textbf{x}})$.\\
In particular, $M/\hspace{-0.15cm}/\bar{\textbf{x}}$ is local Cohen-Macaulay over $A/\hspace{-0.15cm}/\bar{\textbf{x}}$ if and only if $\mathrm{amp}M/\hspace{-0.15cm}/\bar{\textbf{x}}=\mathrm{amp}A/\hspace{-0.15cm}/\bar{\textbf{x}}$.
\end{thm}
\begin{proof} By \cite[Proposition A.4]{BH}, one has the following inequalities:
\begin{center}$(\dag)$\ $\begin{aligned}\mathrm{lc.dim}_AM/\hspace{-0.15cm}/\bar{\emph{\textbf{x}}}
&=\mathrm{sup}\{\mathrm{dim}_{\mathrm{H}^0(A)}\mathrm{H}^\ell(M/\hspace{-0.15cm}/\bar{\emph{\textbf{x}}})+\ell\hspace{0.02cm}|\hspace{0.02cm}\ell\in\mathbb{Z}\}\\
&\geq\mathrm{sup}\{\mathrm{dim}_{\mathrm{H}^0(A)}\mathrm{H}^\ell(M)/\bar{\emph{\textbf{x}}}\mathrm{H}^\ell(M)+\ell\hspace{0.02cm}|\hspace{0.02cm}\ell\in\mathbb{Z}\}\\
&\geq\mathrm{sup}\{\mathrm{dim}_{\mathrm{H}^0(A)}\mathrm{H}^\ell(M)+\ell\hspace{0.02cm}|\hspace{0.02cm}\ell\in\mathbb{Z}\}-n\\
&\geq\mathrm{lc.dim}_AM-n.\end{aligned}$\end{center}As $\mathrm{projdim}_AM<\infty$ if and only if $\mathrm{projdim}_AM/\hspace{-0.15cm}/\bar{\emph{\textbf{x}}}<\infty$, it follows by \cite[Proposition 4.4(3)]{ya20} and \cite[Proposition 2.12]{sh20} that $$\mathrm{projdim}_AM/\hspace{-0.15cm}/\bar{\emph{\textbf{x}}}=\mathrm{supRHom}_A(M/\hspace{-0.15cm}/\bar{\emph{\textbf{x}}},A)=\mathrm{supRHom}_A(M,A/\hspace{-0.15cm}/\bar{\emph{\textbf{x}}}[-n])
=\mathrm{projdim}_AM+n.$$

(1) One has the following (in)equalities\begin{center}$\begin{aligned}\mathrm{grade}_AM/\hspace{-0.15cm}/\bar{\emph{\textbf{x}}}
&\leq\mathrm{lc.dim}A-\mathrm{lc.dim}_AM/\hspace{-0.15cm}/\bar{\emph{\textbf{x}}}+\mathrm{inf}A\\
&\leq\mathrm{lc.dim}A-\mathrm{lc.dim}_AM+n+\mathrm{inf}A\\
&\leq\mathrm{depth}A-\mathrm{depth}_AM+n+\mathrm{inf}A\\
&=\mathrm{projdim}_AM+n+\mathrm{inf}A\\
&=\mathrm{projdim}_AM/\hspace{-0.15cm}/\bar{\emph{\textbf{x}}}+\mathrm{inf}A,\end{aligned}$\end{center}
where the first inequality is by Lemma \ref{lem2.4}(2), the second one is by the inequality $(\dag)$, the third one is by \cite[Corollary 3.3]{Y25}. So $\mathrm{lc.dim}_AM/\hspace{-0.15cm}/\bar{\emph{\textbf{x}}}=\mathrm{lc.dim}_AM-n$.

(2) One has the following (in)equalities\begin{center}$\begin{aligned}\mathrm{grade}_AM/\hspace{-0.15cm}/\bar{\emph{\textbf{x}}}
&=\mathrm{depth}A-\mathrm{lc.dim}_AM/\hspace{-0.15cm}/\bar{\emph{\textbf{x}}}\\
&\leq\mathrm{depth}A-\mathrm{lc.dim}_AM+n\\
&=\mathrm{grade}_AM+n\\
&\leq\mathrm{projdim}_AM+n+\mathrm{inf}A\\
&=\mathrm{projdim}_AM/\hspace{-0.15cm}/\bar{\emph{\textbf{x}}}+\mathrm{inf}A,\end{aligned}$\end{center}
where the first and the second equalities are by Lemma \ref{lem2.4}(2), the first inequality is by  the inequality $(\dag)$ and the second one is by Lemma \ref{lem2.4}(3), as required.

(3) By Nakayama's lemma, one has $\mathrm{sup}M/\hspace{-0.15cm}/\bar{\emph{\textbf{x}}}=\mathrm{sup}M$. By \cite[Proposition 3.17]{sh21} and \cite[Theorem 3.7(1)]{YL}, we have $\mathrm{inf}M/\hspace{-0.15cm}/\bar{\emph{\textbf{x}}}
=\mathrm{lc.dim}_AM-\mathrm{lc.dim}_AM/\hspace{-0.15cm}/\bar{\emph{\textbf{x}}}+\mathrm{inf}M-n$. Thus
we have the following equalities\begin{center}$\begin{aligned}\mathrm{lc.dim}_AM/\hspace{-0.15cm}/\bar{\emph{\textbf{x}}}-\mathrm{sup}M/\hspace{-0.15cm}/\bar{\emph{\textbf{x}}}
&=\mathrm{lc.dim}_AM+\mathrm{inf}M-n-\mathrm{inf}M/\hspace{-0.15cm}/\bar{\emph{\textbf{x}}}-\mathrm{sup}M\\
&=\mathrm{depth}_AM-n-\mathrm{inf}M/\hspace{-0.15cm}/\bar{\emph{\textbf{x}}}\\
&=\mathrm{depth}_AM/\hspace{-0.15cm}/\bar{\emph{\textbf{x}}}-\mathrm{inf}M/\hspace{-0.15cm}/\bar{\emph{\textbf{x}}},\end{aligned}$\end{center}where the third one is by \cite[Lemma 3.6]{YL}, as claimed.
\end{proof}

A map $f:(A,\bar{\mathfrak{m}})\rightarrow (B,\bar{\mathfrak{n}})$ of DG-rings is called \emph{flat local} if $\mathrm{flatdim}_AB=0$ and the induced map $\mathrm{H}^0(f):\mathrm{H}^0(A)\rightarrow\mathrm{H}^0(B)$ is  local.

It is well-known that the tensor product of regular rings is generally not regular. Bouchiba and Kabbaj \cite{BK} proved that if
$R$ is a
flat $k$-module and $S$ is a finitely generated $k$-module, then the tensor product
$R\otimes_kS$ of Cohen-Macaulay rings remains Cohen-Macaulay, where $k$ is a field. As an application of the above theorem, we explore the preservation of sequence-regularity and Cohen-Macaulayness under tensor products in the DG context, which is a DG-version of \cite[Theorem 1]{TY} and improves \cite[Theorem 4.12]{sh21}.

\begin{cor}\label{lem1.6} Let $f:(A,\bar{\mathfrak{m}})\rightarrow (B,\bar{\mathfrak{n}})$ be a flat local homomorphism of DG-rings.

$(1)$ If $B$ is local Cohen-Macaulay with constant amplitude, then $A/\hspace{-0.15cm}/\bar{\mathfrak{m}}\otimes^\mathrm{L}_AB$ is a local Cohen-Macaulay DG-ring.

$(2)$ If $A$ is sequence-regular and $A/\hspace{-0.15cm}/\bar{\mathfrak{m}}\otimes^\mathrm{L}_AB$ is a local Cohen-Macaulay DG-ring, then $B$ is local Cohen-Macaulay.

$(3)$ If $B$ is sequence-regular, then $A$ is sequence-regular.

$(4)$ If $A$ and $A/\hspace{-0.15cm}/\bar{\mathfrak{m}}\otimes^\mathrm{L}_AB$ are sequence-regular, then $B$ is sequence-regular.
\end{cor}
\begin{proof} As $\mathrm{flatdim}_AB=0$, the induced map $\mathrm{H}^0(f):\mathrm{H}^0(A)\rightarrow\mathrm{H}^0(B)$ is also flat.

(1) This follows from Theorem \ref{lem2.3}(3) and \cite[Proposition 2.9]{sh20}.

(2) By assumption, $\bar{\mathfrak{m}}$ is generated by an $A$-regular sequence $\bar{x}_1,\cdots,\bar{x}_d$ with $d=\mathrm{lc.dim}A$, so $A/\hspace{-0.15cm}/\bar{\mathfrak{m}}\otimes^\mathrm{L}_AB\simeq B/\hspace{-0.15cm}/(f(\bar{x_1}),\cdots,f(\bar{x_d}))$. Since $\mathrm{flatdim}_AB=0$,  it follows by \cite[Lemma 2.13(2)]{Mi19} and \cite[Lemma 4.11]{sh21} that $\mathrm{inf}(A/\hspace{-0.15cm}/\bar{\mathfrak{m}}\otimes^\mathrm{L}_AB)=\mathrm{inf}A/\hspace{-0.15cm}/(\bar{x}_1,\cdots,\bar{x}_d)=\mathrm{inf}A=\mathrm{inf}B$,
 so $f(\bar{x_1}),\cdots,f(\bar{x_d})\in\bar{\mathfrak{n}}$ is $B$-regular by \cite[Lemma 2.13(2)]{Mi19} again. Hence \cite[Theorem A.11]{BH} implies that $\mathrm{lc.dim}B=\mathrm{dim}\mathrm{H}^0(B)
=\mathrm{dim}\mathrm{H}^0(A/\hspace{-0.15cm}/\bar{\mathfrak{m}}\otimes^\mathrm{L}_AB)+\mathrm{dim}\mathrm{H}^0(A)=
\mathrm{depth}(A/\hspace{-0.15cm}/\bar{\mathfrak{m}}\otimes^\mathrm{L}_AB)-\mathrm{inf}(A/\hspace{-0.15cm}/\bar{\mathfrak{m}}\otimes^\mathrm{L}_AB)+\mathrm{lc.dim}A=
\mathrm{depth}B/\hspace{-0.15cm}/(f(\bar{x_1}),\cdots,f(\bar{x_d}))-\mathrm{inf}B+\mathrm{lc.dim}A
=\mathrm{depth}B-\mathrm{inf}B$ and then $B$ is local Cohen-Macaulay.

(3) By \cite[Theorem 3.4]{sh21}, $\mathrm{H}^0(B)$ is regular, it follows by \cite[Theorem 1(b1)]{TY} that $\mathrm{H}^0(A)$ is regular and $\bar{\mathfrak{m}}$ is generated by an $\mathrm{H}^0(A)$-regular sequence $\bar{x}_1,\cdots,\bar{x}_d$ with $d=\mathrm{lc.dim}A$. By \cite[Proposition 1.1.2]{BH}, one has $f(\bar{x}_1),\cdots,f(\bar{x}_d)$ is $\mathrm{H}^0(B)$-regular. As $B$ is sequence-regular, $B$ has constant amplitude by \cite[Remark 7.4]{sh21}, so $f(\bar{x}_1),\cdots,f(\bar{x}_d)$ is $B$-regular by \cite[Theorem 3.3]{sh21}. Thus
$A/\hspace{-0.15cm}/\bar{\mathfrak{m}}\otimes^\mathrm{L}_AB\simeq B/\hspace{-0.15cm}/f(\bar{x_1}),\cdots,f(\bar{x_d})$ is a local Cohen-Macaulay DG-ring by Theorem \ref{lem2.3}(3), and we have the following (in)equalities
\begin{center}$\begin{aligned}\mathrm{depth}B-\mathrm{inf}B-\mathrm{depth}A+\mathrm{inf}A
&\leq\mathrm{dim}\mathrm{H}^0(B)-\mathrm{dim}\mathrm{H}^0(A)\\
&=\mathrm{dim}\mathrm{H}^0(A/\hspace{-0.15cm}/\bar{\mathfrak{m}}\otimes^\mathrm{L}_AB)\\
&=\mathrm{lc.dim}(A/\hspace{-0.15cm}/\bar{\mathfrak{m}}\otimes^\mathrm{L}_AB)\\
&=\mathrm{depth}B/\hspace{-0.15cm}/f(\bar{x_1}),\cdots,f(\bar{x_d})-\mathrm{inf}B/\hspace{-0.15cm}/f(\bar{x_1}),\cdots,f(\bar{x_d})\\
&=\mathrm{depth}B-d-\mathrm{inf}B\\
&\leq\mathrm{depth}B-\mathrm{inf}B-\mathrm{depth}A+\mathrm{inf}A,\end{aligned}$\end{center}
where the two inequalities are by \cite[Corollary 5.5]{Shaul}, the fourth equality is by \cite[Proposition 3.2]{sh20} and \cite[Lemma 2.13(2)]{Mi19}. Thus $\mathrm{lc.dim}A=\mathrm{depth}A-\mathrm{inf}A$, and then $A$ is sequence-regular by \cite[Theorem 3.4]{sh21}.

(4) Since $f:(A,\bar{\mathfrak{m}})\rightarrow (B,\bar{\mathfrak{n}})$ be a flat local homomorphism of DG-rings, the map $\mathrm{L}\Lambda_{\bar{\mathfrak{m}}}(A,\bar{\mathfrak{m}})\rightarrow \mathrm{L}\Lambda_{\bar{\mathfrak{n}}}(B,\bar{\mathfrak{n}})$ is flat local by \cite[Proposition 3.18]{s19}. By \cite[Corollary 4.5]{sh21}, we may assume that $A$ and $B$ are derived completion. Then $A$ and $B$ have dualizing DG-modules by \cite[Proposition 7.21]{s18}.
As $\mathrm{H}^0(A/\hspace{-0.15cm}/\bar{\mathfrak{m}}\otimes^\mathrm{L}_AB)\cong \mathrm{H}^0(A)/\bar{\mathfrak{m}}\otimes_{\mathrm{H}^0(A)}\mathrm{H}^0(B)$ is regular by \cite[Theorem 3.4]{sh21}, it follows by \cite[Theorem 4.12(2)]{sh21} that $B$ is sequence-regular.
\end{proof}

\bigskip \centerline {\bf Acknowledgements}
\bigskip
This research was partially supported by National Natural Science Foundation of China (12571035).

\renewcommand\refname

{\bf References}
\begin{thebibliography}{99}
\bibitem{AY}T. Araya, Y. Yoshino,  Remarks on a depth formula, a grade inequality and a conjecture of Auslander, \emph{Comm. Algebra} \textbf{11} (1998) 3793--3806.
\bibitem{BK}S. Bouchiba, S.E. Kabbaj, Tensor products of Cohen-Macaulay rings: solution to a problem of Grothendieck, \emph{J. Algebra} \textbf{252} (2002) 65--73.
\bibitem{B}N. Bourbaki, \emph{Alg$\grave{\mathrm{e}}$bre Commutative}, Herman, Paris 1967.
\bibitem{BH}W. Bruns, J. Herzog, \emph{Cohen-Macaulay rings}, Cambridge Studies in Advanced Mathematics 39, Cambridge University Press, 1993.
\bibitem{BSSW}I. Bird, L. Shaul, P. Sridhar, J. Williamson, Finitistic dimension over commutative DG-rings, \emph{Math. Z.} \textbf{309} (2025) 1--29.
\bibitem{CFH}L.W. Christensen, H.-B. Foxby, H. Holm. \emph{Derived Category Methods in Commutative Algebra}, Springer Monographs in Mathematics, 2024.
\bibitem{DAT}K.  Divaani-Aazar, F.M. Aghjeh Mashhad, M. Tousi, On the existence of certain modules of finite Gorenstein homological dimensions, \emph{Comm. Algebra} \textbf{42} (2013) 1630--1643.
\bibitem{F}H. B. Foxby,  Quasi-perfect modules over Cohen-Macaulay rings, \emph{Math. Nachr.} \textbf{66} (1975) 103--110.
\bibitem{FIJ}A. Frankild, S. Iyengar, P. J${\o}$gensen, Dualizing differential graded modules and Gorenstein differential graded algebras, \emph{J. London Math. Soc.} \textbf{68} (2003) 288--306.
\bibitem{JMM}V.H. Jorge-P$\acute{\mathrm{e}}$rez, P. Martins, V.D. Mendoza-Rubio, Ischebeck's formula, grade and quasi-homological dimensions, arXiv: 2412.06659v1, 2024.
\bibitem{KY}L. Khatami, S. Yassemi, Cohen-Macaulayness of tensor products, \emph{Rocky Mountain J. Math.} \textbf{34} (2004) 205--213.
\bibitem{M}H. Matsumura, \emph{Commutative ring theory}, Combridge University press, 1986.
\bibitem{Mi19}H. Minamoto, Homological identities and dualizing complexes of commutative differential graded algebras, \emph{Israel J. Math.} \textbf{242} (2021) 1--36.
\bibitem{R}D. Rees, The Grade of an ideal or module, \emph{Proc. Camb. Phil. Soc.} \textbf{53} (1957) 28--42.
\bibitem{s18}L. Shaul, Injective DG-modules over non-positive DG-rings, \emph{J. Algebra} \textbf{515} (2018) 102--156.
\bibitem{s19}L. Shaul,  Completion and torsion over commutative DG rings, \emph{Israel J. Math.} \textbf{232} (2019) 531--588.
\bibitem{Shaul}L. Shaul, The Cohen-Macaulay property in derived commatative algebra, \emph{Trans. Amer. Math. Soc.} \textbf{373} (2020) 6095--6138.
\bibitem{sh20}L. Shaul, Koszul complexes over Cohen-Macaulay rings, \emph{Adv. Math.} \textbf{386} (2021) 107806.
\bibitem{sh21} L. Shaul, Sequence-regular commutative DG-rings, \emph{J. Algebra} \textbf{647} (2024) 400--435.
\bibitem{TY}M. Tousi, S. Yassemi, Tensor products of some special rings, \emph{J. Algebra} \textbf{268} (2003) 672--676.
\bibitem{Y25} X.Y. Yang, Intersection Theorem for DG-modules over non-positive DG-rings, \emph{Proc. Roy. Soc. Edinb. A}, published online.
\bibitem{ya20}X.Y. Yang, L. Wang, Homological invariants over non-positive DG-rings, \emph{J. Algebra Appl.} 2050153 (2020) (18 pages).
\bibitem{YL} X.Y. Yang, Y.J. Li, Local Cohen-Macaulay DG-Modules, \emph{Appl. Categor. Struct.} \textbf{31} (2023) 1--13.
\bibitem{Ya}S. Yassemi, L. Khatami, T. Sharif, Grade and Gorenstein dimension, \emph{Comm. Algebra} \textbf{29} (2001) 5085--5094.
\bibitem{Y}A. Yekutieli, \emph{Derived Categories}, Cambridge Studies in Advanced Mathematics, Cambridge University Press, Cambridge, 2019.
\bibitem{Yo}K. Yoshida, Tensor products of perfect modules and maximal surjective Buchsbaum modules, \emph{J. Pure Appl. Algebra} \textbf{123} (1998) 313--326.
\end{thebibliography}
\end{document}